\documentclass[11pt,a4paper]{amsart}
\usepackage{a4wide}
\usepackage{subfiles}
\usepackage{lipsum}
\usepackage[hidelinks]{hyperref}

\usepackage{enumitem}
\usepackage{amsmath}
\usepackage{amssymb}
\usepackage{mathtools}
\usepackage{amsfonts}
\usepackage{microtype}
\usepackage{multicol}
\usepackage{subcaption}
\usepackage{xcolor}
\usepackage{stackengine,scalerel}
\usepackage{tikz}
\usepackage{csquotes}
\usepackage{tikz-cd}
\usepackage{amsthm}
\usepackage{tcolorbox}
\usepackage[capitalise]{cleveref}
\usepackage[utf8]{inputenc}
\usepackage[english]{babel}

\DeclareMathAlphabet{\mathpzc}{OT1}{pzc}{m}{it}

\usepackage{newunicodechar}
\newunicodechar{⟨}{\langle}
\newunicodechar{⟩}{\rangle}

\usepackage[style=alphabetic,backend=biber,maxnames=5,maxalphanames=5,language=english,doi=false,isbn=false,url=true]{biblatex}
\bibliography{literature.bib}
  
\renewbibmacro{in:}{}
\renewbibmacro*{issue+date}{%
  \ifboolexpr{not test {\iffieldundef{year}} or not test {\iffieldundef{issue}}}
    {\printtext[parens]{%
       \iffieldundef{issue}
         {\usebibmacro{date}}
         {\printfield{issue}%
          \setunit*{\addspace}%
          \usebibmacro{date}}}}
    {}%
  \newunit}
\let\oldtocsection=\tocsection
\let\oldtocsubsection=\tocsubsection
\let\oldtocsubsubsection=\tocsubsubsection
\renewcommand{\tocsection}[2]{\hspace{0em}\oldtocsection{#1}{#2}}
\renewcommand{\tocsubsection}[2]{\hspace{1em}\oldtocsubsection{#1}{#2}}
\renewcommand{\tocsubsubsection}[2]{\hspace{2em}\oldtocsubsubsection{#1}{#2}}

\newtheorem{thm}{Theorem}
\newtheorem{lem}[thm]{Lemma}

\theoremstyle{definition}

\newtheorem*{nrem}{Remark}

\setenumerate[1]{leftmargin=*,labelindent=0pt,label=(\roman*),}

\newcommand{\pref}[2]{\hyperref[#2]{#1 \ref*{#2}}}
\usepackage{todonotes}

\usepackage[T1]{fontenc}
 \usepackage{XCharter}

\begin{document}

\title{Side lengths of cubes with vertices in $\mathbb Z^n$}

\begin{abstract} 
We determine the set of side lengths of $d$-dimensional cubes with vertices in $\mathbb Z^n$ using Witt's cancellation theorem from the algebraic theory of quadratic forms. 

\end{abstract}

\author[C.~Bernert]{Christian Bernert}

\address{Institute of Science and Technology Austria, Am Campus 1, 3400 Klosterneuburg, Austria}

\email{christian.bernert@ist.ac.at}

\author[J.~Reinhold]{Jens Reinhold}

\address{Sankt Augustin, Germany}

\email{jens.reinhold@posteo.de}

\subjclass[2020]{11E25 (11E81, 52B20)}

\maketitle
\section{Background}

In the 18th century, a central problem in number theory was to determine the set 
\[I_{n} = \left\{\sum_{j = 1}^{n} a_j^2 \ | \ a_1, \dots, a_{n} \in \mathbb Z\right\}\]
of integers that can be expressed as a sum of $n$ perfect squares. It was resolved by three advances that shaped the development of modern number theory:

\begin{enumerate}
    \item Fermat observed that $m \in I_2$ if and only if each prime $p \equiv 3 \pmod{4}$ divides $m$ an even number of times. This result is known as Fermat's theorem on sums of two squares today, even though it was Euler who first published a proof in 1747. 
    \item In 1770, Lagrange proved that every non-negative integer is a sum of four squares, showing that $I_4=\mathbb{N}_0$. Hence also $I_n=\mathbb{N}_0$ for all $n \ge 4$. 
    \item Finally, sums of three squares turned out to be the hardest case. In 1797, Legendre proved that $m \in I_3$ unless $m=4^k\ell$ with $\ell \equiv 7 \pmod{8}$.
\end{enumerate}

We refer the reader to \cite[Chapter~XX]{Hardy-Wright} for a discussion of these classical results. We remark that Fermat's theorem led to the rich algebraic theory of binary quadratic forms, intimately linked to the factorization of integers in quadratic number fields, with its generalizations anticipating modern class field theory. Meanwhile, by the work of Jacobi, Lagrange's theorem can be linked to the fact that the associated theta series is a modular form, thus providing an early and important example for the theory of automorphic forms. While sums of two squares as norms of Gaussian integers and sums of four squares as norms of integral quaternions inherit a useful multiplicative structure, this is not the case for sums of three squares. Nonetheless, as observed by Gauß, one can essentially identify representations of $n$ as a sum of three squares with classes of binary quadratic forms of discriminant $-n$.

\section{A geometric generalization}

Geometrically, describing $I_n$ amounts to determining all possible lengths of vectors in $\mathbb{Z}^n$. This point of view leads to a generalization: What happens if we replace vectors by (hyper)cubes? To this end, for a given positive integers $d \le n$, let us denote by $J(d,n)$ the set of squares of side lengths of $d$-dimensional hypercubes with vertices in $\mathbb{Z}^n$. In particular, note that  $J(1,n) = I_n$ and $J(2,2)=J(1,2)=I_2$: indeed, any integer vector $(a,b) \in \mathbb{Z}^2$ can be extended to a square in the plane by the vector $(-b,a)$ as shown in Figure \ref{fig:j22}.

\bigskip
\begin{figure}[h]
\begin{center}
    \begin{tikzpicture}[scale=0.5]

    \draw[step=1, gray!40, thin] (-3,-1) grid (4,6);

    \draw[->, thick] (-3,0) -- (4,0);
    \draw[->, thick] (0,-1) -- (0,6);

    \coordinate (O) at (0,0);
    \coordinate (A) at (3,2);
    \coordinate (B) at (1,5); 
    \coordinate (C) at (-2,3);

    \draw[thick, blue] (O) -- (A) -- (B) -- (C) -- cycle;

    \draw[->, very thick, red] (O) -- (A) node[pos=0.7, below right] {$(a,b)$};
    \draw[->, very thick, red] (O) -- (C) node[pos=1.1, above] {$(-b,a)$};

    \fill (O) circle (3pt);

    \foreach \x in {-3,...,4}
        \foreach \y in {-1,...,6}
            \fill (\x,\y) circle (1pt);

    \end{tikzpicture} 
    \caption{}\label{fig:j22}
\end{center}
\end{figure}

On the other hand, note that $3=1^2+1^2+1^2 \in J(1,3)=I_3$, but there is no way of extending any of the witnessing vectors $(\pm 1,\pm 1, \pm 1)$ to a square (let alone a cube) in $\mathbb{Z}^3$, showing that $3 \not \in J(2,3)$ and hence $J(2,3) \subsetneq J(1,3)$. We can now state our main result:

\begin{thm} \label{thm:main} For positive integers $d \leq n$, the set $J(d,n)$ is given by the entries of the following table.
\begin{center}
\begin{tabular}{ c | c c c c}
 & $d \text{ mod } 4$ & $\quad \quad$ & $ \quad \quad$ & $\quad \quad$\\ 
n-d  & 1 & 2 & 3 & $0$ \\ 
\hline
0   & $I_1$ & $I_2$ & $I_1$ &  $\mathbb N_0$\\ 
1   & $I_2$ & $I_2$ & $\mathbb N_0$ &  $\mathbb N_0$ \\ 
2   & $I_3$ & $\mathbb N_0$  & $\mathbb N_0$ &  $\mathbb N_0$ \\ 
$\geq 3$   & $\mathbb N_0$  & $\mathbb N_0$  & $\mathbb N_0$ &  $\mathbb N_0$ \\ 
\end{tabular}
\end{center}
\end{thm}

Here of course $I_1$ is just the set of perfect squares itself, while $I_2$ and $I_3$ are described by Fermat's and Legendre's results in the previous section. Note that $J(d,n)$ only depends on $n-d$ and the value of $d$ modulo $4$.

\begin{nrem} The first row of this chart was determined by  S.~Biderman and N.~Elkies in a remarkable post on Mathoverflow \cite{MO-post} that inspired this note.
\end{nrem}

\section{Proofs}
As suggested by Elkies, our argument will use the algebraic theory of quadratic forms, and specifically Witt's cancellation theorem. Since this becomes much simpler when working over a field rather than a ring, it is natural to study the set $J^{\mathbb{Q}}(d,n)$ of integers which appear as squares of the side lengths of $d$-dimensional cubes with vertices in $\mathbb{Q}^n$.

Indeed, our argument will show that $J^{\mathbb{Q}}(d,n)=J(d,n)$ for all $(d,n)$, but we are not aware of any \textit{a priori} reason to deduce this, only the inclusion $J(d,n) \subset J^{\mathbb{Q}}(d,n)$ is immediate.

We begin by recalling some facts and notation from the algebraic theory of quadratic forms over $\mathbb{Q}$. Identifying an $n$-ary quadratic form with a symmetric matrix $Q \in \mathbb{Q}^{n \times n}$, we say that two such forms $Q_1,Q_2$ are equivalent if $Q_2=A^tQ_1A$ for some invertible matrix $A \in \mathbb{Q}^{n \times n}$. We write $Q_1 \simeq Q_2$.
We can also form the direct sum of two quadratic forms: If $Q_1 \in \mathbb{Q}^{n_1 \times n_1}$ and $Q_2 \in \mathbb{Q}^{n_2 \times n_2}$, then we put $(Q_1 \oplus Q_2)(\mathbf{x},\mathbf{y}):=Q_1(\mathbf{x})+Q_2(\mathbf{y})$ where $(\mathbf{x},\mathbf{y}) \in \mathbb{Q}^{n_1+n_2}$. In terms of matrices this corresponds to forming the block matrix $\begin{pmatrix} Q_1 & 0\\ 0 & Q_2 \end{pmatrix}$. We denote a $k$-fold sum by $k \cdot Q:=Q \oplus Q \oplus \dots \oplus Q$. Finally, we denote by $\langle a_1,\dots,a_n\rangle=\langle a_1\rangle \oplus \dots \oplus \langle a_n\rangle$ the diagonal quadratic form $Q(x_1,\dots,x_n)=\sum_{i=1}^n a_ix_i^2$ represented by the matrix $\text{diag}(a_1,\dots,a_n)$.

We can now describe the relation between these notions and our problem. Indeed, suppose that we are given an orthogonal basis of vectors $v_1,\dots,v_n \in \mathbb{Q}^n$. Then the matrix formed by these vectors describes an equivalence between the $n$-ary quadratic forms $n \cdot \langle 1\rangle=\langle 1,1,\dots,1\rangle$ and $\langle \|v_1\|^2,\dots,\|v_n\|^2\rangle$. Conversely, if the forms are equivalent, we may obtain such an orthogonal basis from the transformation matrix.

Henceforth, our interest is shifted to a criterion for equivalence of forms. This is given by the following result; for a proof see e.g. \cite{Pfister}, Chapter 2.

\begin{lem}[Witt's cancellation theorem] \label{thm:witt}
Let $\varphi,\varphi_1,\varphi_2$ be quadratic forms over $\mathbb{Q}$ such that $\varphi \oplus \varphi_1 \simeq \varphi \oplus \varphi_2$. Then $\varphi_1 \simeq \varphi_2$.
\end{lem}

We prepare the stage for proving \cref{thm:main} by studying the special case $d=n=4$ which will serve as a vehicle to translate between $J(d,n)$ and $J(d+4,n+4)$ later on.

\begin{lem}\label{lem:fourdim}
We have $J(4,4)=\mathbb{N}$. In particular, for any $m \ge 1$, the quadratic forms $4 \cdot \langle m\rangle$ and $4 \cdot \langle 1\rangle$ are equivalent.
\begin{proof}
Let $m$ be a positive integer. By Lagrange's theorem, we may write $m=a^2+b^2+c^2+d^2$ with integers $a,b,c,d$. Now one checks that the rows of the matrix
\[\begin{pmatrix} a & b & c & d\\-b & a & -d & c\\-c & -d & a & b\\-d & c & -b & a \end{pmatrix}\]
describe an orthogonal basis of $\mathbb{Z}^4$ consisting of integer vectors of length $\sqrt{m}$.
Indeed, one may motivate this construction by viewing the vector $(a,b,c,d)$ as the quaternion $q=a+bi+cj+dij$ and then considering the basis $(q,qi,qj,qk)$.
\end{proof}
\end{lem}

Equipped with this observation, we are able to prove our key lemma which reduces the proof of the main result to a handful of cases.

\begin{lem}
For $n \ge d \ge 1$, we have $J^{\mathbb{Q}}(d,n)=J^{\mathbb{Q}}(d+4,n+4)$. Moreover, $J(d,n) \subset J(d+4,n+4)$.
\begin{proof}
Suppose that $m \in J^{\mathbb{Q}}(d,n)$, so we are given $d$ orthogonal vectors in $\mathbb{Q}^n$ of length $\sqrt{m}$. Taking the direct sum with the matrix from \cref{lem:fourdim}, we may extend these to $d+4$ orthogonal vectors in $\mathbb{Q}^{n+4}$.
Hence $m \in J^{\mathbb{Q}}(d+4,n+4)$ as desired. Moreover, starting with vectors in $\mathbb{Z}^n$, the construction in fact yields vectors in $\mathbb{Z}^{n+4}$, so that $J(d,n) \subset J(d+4,n+4)$.

It remains to show that $J^{\mathbb{Q}}(d+4,n+4) \subset J^{\mathbb{Q}}(d,n)$. Starting with $d+4$ orthogonal vectors in $\mathbb{Q}^{n+4}$, we can extend them arbitrarily to an orthogonal basis of $\mathbb{Q}^{n+4}$. By our previous comments, this means that
\[(d+4) \cdot \langle m\rangle\oplus \langle \ell_1,\dots,\ell_{n-d}\rangle \simeq (n+4) \cdot \langle 1\rangle\]
for certain rational numbers $\ell_1,\dots,\ell_{n-d}$. By \cref{lem:fourdim}, we know that $4 \cdot \langle m\rangle \simeq 4 \cdot \langle 1\rangle$, so Witt's cancellation theorem implies that
\[d \cdot \langle m\rangle\oplus \langle \ell_1,\dots,\ell_{n-d}\rangle \simeq n \cdot \langle 1\rangle.\]
Translating back, we then find an orthogonal basis of $\mathbb{Q}^n$ containing $d$ vectors of length $\sqrt{m}$, so that indeed $m \in J^{\mathbb{Q}}(d,n)$.
\end{proof}
\end{lem}

\textit{Proof of Theorem~\ref{thm:main}:} By \cref{lem:fourdim}, it suffices to show that for $d \le 4$, we have $J(d,n)=J^{\mathbb{Q}}(d,n)$ and that this set is given by the corresponding entry in the table.

Moreover, since $J(4,4)=\mathbb{N}$ it is clear that $J(d,n)=\mathbb{N}$ whenever $d \le 4 \le n$ so that we may restrict to $d,n \le 3$.
If $d=1$ and $n \in \{1,2,3\}$ we have $J(1,n)=I_n$ by definition. Moreover, it is an easy consequence of the characterization of $I_2$ and $I_3$ that indeed $J^{\mathbb{Q}}(1,n)=I_n$.

If $d=n=3$, it is clear that $I_1 \subset J(3,3)$. Conversely, any cube (indeed any polyhedron) whose vertices have rational coordinates must have rational volume, hence if $m \in J^{\mathbb{Q}}(3,3)$, then $(\sqrt{m})^3$ must be rational and hence $m \in I_1$ as desired.

Finally, if $d=2$ it is clear that $I_2 = J(2,2) \subset J(2,3) \subset J^{\mathbb{Q}}(2,3)$. So we are left to prove that $J^{\mathbb{Q}}(2,3) \subset I_2$.
This requires one final trick: Suppose that $m \in J^{\mathbb{Q}}(2,3)$ so that we are given two orthogonal vectors of length $\sqrt{m}$ in $\mathbb{Q}^3$. We argue as in the proof of \cref{lem:fourdim}, but this time we choose a specific extension to an orthogonal basis. More precisely, we may choose the third vector to be the cross product of the two initial two vectors, having length $m$. We conclude that
\[\langle m,m,m^2\rangle \simeq \langle 1,1,1\rangle.\]
Noting that $\langle m^2\rangle \simeq \langle 1\rangle$, we can apply Witt cancellation again and find that
\[\langle m,m\rangle \simeq \langle 1,1\rangle\]
so that, in particular, $m$ is represented by $\langle 1,1\rangle$ and hence a sum of two rational squares, and then by our previous remarks indeed a sum of two integer squares. \qed

\begin{nrem} Let us note that the use of Witt's theorem can be transformed into a constructive version of our result, at least when rational coordinates are allowed. For instance, in the case of $J(2,3)$, let us follow through our argument: One finds that from two perpendicular vectors $(a,b,c)$ and $(d,e,f)$ of length $\sqrt{m}$, i.e. a solution of the system
\begin{align*}
    a^2+b^2+c^2 &=m\\
    d^2+e^2+f^2 &=m\\
    ad+be+cf &=0
\end{align*}
we obtain, using the cross-product as a third basis vector, the identity
\[(ax+dy+(bf-ce)z)^2+(bx+ey+(cd-af)z)^2+(cx+fy+(ae-bd)z)^2=mx^2+my^2+(mz)^2\]
of ternary quadratic forms. Witt's theorem then amounts to the observation that we may enforce $cx+fy+(ae-bd)z=mz$ by choosing $z=\frac{cx+fy}{bd-ae-m}$, where we may assume that $bd-ae \ne m$ by changing the signs if necessary. We may then choose $x=1, y=0$ to obtain the representation 
\[m=\left(\frac{m(d-b)}{bd-ae-m}\right)^2+\left(\frac{m(a+e)}{bd-ae-m}\right)^2\]
of $m$ as a sum of two rational squares.
\end{nrem}

\printbibliography

@article {MO-post,
    AUTHOR = {Biderman, S. and Elkies, N.},
     TITLE = {Characterization of Volumes of Lattice Cubes},
   JOURNAL = {Mathoverflow},
   YEAR = {2014},
       URL = {https://mathoverflow.net/questions/161141/characterization-of-volumes-of-lattice-cubes},
}

@book {Pfister,
    AUTHOR = {Pfister, Albrecht},
     TITLE = {Quadratic forms with applications to algebraic geometry and
              topology},
    SERIES = {London Mathematical Society Lecture Note Series},
    VOLUME = {217},
 PUBLISHER = {Cambridge University Press, Cambridge},
      YEAR = {1995},
     PAGES = {viii+179},
}

@book {Hardy-Wright,
    AUTHOR = {Hardy, G. H. and Wright, E. M.},
     TITLE = {An introduction to the theory of numbers},
   EDITION = {Sixth},
      NOTE = {Revised by D. R. Heath-Brown and J. H. Silverman,
              With a foreword by Andrew Wiles},
 PUBLISHER = {Oxford University Press, Oxford},
      YEAR = {2008},
     PAGES = {xxii+621},
      ISBN = {978-0-19-921986-5},
   MRCLASS = {11-01},
  MRNUMBER = {2445243},
}
\nopagebreak
\end{document}